\documentclass{amsart}
\usepackage{tikz,amssymb,longtable,booktabs}
\usetikzlibrary{arrows}
\theoremstyle{plain}
\newtheorem{theorem}{Theorem}
\newtheorem{proposition}[theorem]{Proposition}
\newtheorem{lemma}[theorem]{Lemma}
\theoremstyle{definition}
\newtheorem{definition}[theorem]{Definition}
\newtheorem{example}[theorem]{Example}

\newcommand{\cat}[1]{\ensuremath{\mathbf{#1}}}
\newcommand{\id}[1][]{\ensuremath{\mathrm{id}_{#1}}}
\newcommand{\op}{\ensuremath{^\textrm{\rm op}}}
\DeclareMathOperator{\Tr}{Tr}
\DeclareMathOperator{\tr}{tr}
\DeclareMathOperator{\dom}{dom}
\newcommand{\To}{\ensuremath{\Rightarrow}} 
\newcommand{\TO}{\ensuremath{\Rrightarrow}} 
\tikzstyle{twocell}=[-implies,double equal sign distance]
\tikzstyle{threecell}=[draw,shorten <=.5pt,shorten >=.5pt,-implies,preaction={draw,double distance=2.5pt,-implies}]

\begin{document}
\title{Compact inverse categories}
\author{Robin Cockett}
\address{University of Calgary, 2500 University Drive NW, Calgary AB T2N 1N4, Canada}
\email{robin@ucalgary.ca}
\author{Chris Heunen}
\address{University of Edinburgh, Informatics Forum, Edinburgh EH8 9AB, United Kingdom}
\email{chris.heunen@ed.ac.uk}
\thanks{Supported by EPSRC Fellowship EP/R044759/1. We thank Peter Hines for pointing out that the proof of Proposition~\ref{prop:oneobject} could be simplified, Martti Karvonen for the idea of the proof of Lemma~\ref{lem:samediff}, and Phil Scott for pointing out Theorem~\ref{thm:jarek}.}
\date{\today}
\begin{abstract}
  The Ehresmann-Schein-Nambooripad theorem gives a structure theorem for inverse monoids: they are inductive groupoids.
  A particularly nice case due to Jarek is that commutative inverse monoids become semilattices of abelian groups.
  It has also been categorified by DeWolf-Pronk to a structure theorem for inverse categories as locally complete inductive groupoids.
  We show that in the case of compact inverse categories, this takes the particularly nice form of a semilattice of compact groupoids. 
  Moreover, one-object compact inverse categories are exactly commutative inverse monoids.
  Compact groupoids, in turn, are determined in particularly simple terms of 3-cocycles by Baez-Lauda.
\end{abstract}
\maketitle

\section{Introduction}

Inverse monoids model partial symmetry~\cite{lawson:inversesemigroups}, and arise naturally in many combinatorial constructions~\cite{duncanpaterson:cstaralgebras}.
The easiest example of an inverse monoid is perhaps a group. 
There is a structure theorem for inverse monoids, due to Ehresmann-Schein-Nambooripad~\cite{ehresmann:gattungen,ehresmann:inductives,schein:inversesemigroups,nambooripad:regularsemigroups}, that exhibits them as inductive groupoids.
The latter are groupoids internal to the category of partially ordered sets with certain extra requirements.
By a result of Jarek~\cite{jarek:semigroups}, the inductive groupoids corresponding to commutative inverse monoids can equivalently be described as semilattices of abelian groups.

A natural typed version of an inverse monoid is an inverse category~\cite{kastl:inversecategories,cockettlack:restriction}. This notion can for example model partial reversible functional programs~\cite{giles:thesis}.
The easiest example of an inverse category is perhaps a groupoid.
DeWolf-Pronk have generalised the ESN theorem to inverse categories, exhibiting them as locally complete inductive groupoids.
This paper investigates `the commutative case', thus fitting in the bottom right cell of Figure~\ref{fig:overview}.

\begin{figure}[h]\label{fig:overview}
  \begin{tabular}{lll}
    \toprule
    objects & general case & commutative case \\
    \midrule 
    one & inductive groupoid~\cite{nambooripad:regularsemigroups} & semilattice of abelian groups~\cite{jarek:semigroups} \\
    many & locally inductive groupoid~\cite{dewolfpronk:esntheorem} & semilattice of compact groupoids \\
    \bottomrule
  \end{tabular}
  \caption{Overview of structure theorems for inverse categories.}
\end{figure}

However, let us emphasise two ways in which Figure~\ref{fig:overview} is overly simplified.
First, the term `commutative case' is misleading: we mean considering compact inverse categories. 
More precisely, we prove that compact inverse categories correspond to semilattices of compact groupoids. 
Compact inverse categories are only commutative in that their endohomset of scalars is always commutative. 
In particular, the categorical composition of the compact inverse category can be as noncommutative as you like.
We expect that the tensor product also need not be symmetric.
But compact categories are interesting in their own right: 
they model quantum entanglement~\cite{heunenvicary:cqm}; they model linear logic~\cite{seely:linearlogic}; and they naturally extend traced monoidal categories modelling feedback~\cite{joyalstreetverity:traced}.

Second, our result is not a straightforward special case of DeWolf-Pronk~\cite{dewolfpronk:esntheorem}, nor of Jarek~\cite{jarek:semigroups}, but instead rather a common categorification.
We prove that one-object compact inverse categories are exactly commutative inverse monoids.
Semilattices of groupoids are a purely categorical notion, whereas ordered groupoids have more ad hoc aspects. 
Compact groupoids are also known as 2-groups or crossed modules, and have fairly rigid structure themselves, due to work by Baez and Lauda~\cite{baezlauda:2groups}.
We take advantage of this fact to ultimately show that there is a (weak) 2-equivalence of (weak) 2-categories of compact inverse categories, and semilattices of 3-cocycles.

Section~\ref{sec:inversemonoids} starts by recalling the ESN structure theorem for inverse monoids, and its special commutative case due to Jarek in a language that the rest of the paper will follow. 
Section~\ref{sec:inversecategories} discusses the generalisation of the ESN theorem to inverse categories due to DeWolf and Pronk, and its relation to semilattices of groupoids.
Section~\ref{sec:compactinversecategories} is the heart of the paper, and considers additional structure on inverse categories that was hidden for inverse monoids. It shows that the construction works for compact inverse categories, and argues that this is the right generalisation of inverse monoids in this sense. 
After all this theory, Section~\ref{sec:examples} lists examples. We have chosen to treat examples after theory; that way they can illustrate not just compact inverse categories, but also the construction of the structure theorem itself. 
Section~\ref{sec:compactgroupoids} then moves to a 2-categorical perspective, to connect to the structure theorem for compact groupoids due to Baez and Lauda.
Finally, Section~\ref{sec:conclusion} discusses the many questions left open and raised in the paper.

\section{Inverse monoids}\label{sec:inversemonoids}

An \emph{inverse monoid} is a monoid where every element $x$ has a unique element $x^\dag$ satisfying $x=xx^\dag x$ and $x^\dag=x^\dag x x^\dag$~\cite{lawson:inversesemigroups}. Equivalently, the monoid carries an involution $\dag$ such that $x=xx^\dag x$ and $xx^\dag yy^\dag=yy^\dag y xx^\dag$ for all elements $x$ and $y$. Inverse monoids and involution-respecting homomorphisms form a category $\cat{InvMon}$, and commutative inverse monoids form a full subcategory $\cat{cInvMon}$. This section recalls structure theorems for inverse monoids. In general they correspond to inductive groupoids by the Ehresmann-Schein-Nambooripad theorem~\cite{ehresmann:gattungen,ehresmann:inductives,nambooripad:regularsemigroups,schein:inversesemigroups}, that we now recall. 

\begin{definition}\label{def:semilattice}
  A \emph{(bounded meet-)semilattice} is a partially ordered set with a greatest element $\top$, in which any two elements $s$ and $t$ have a greatest lower bound $s \wedge t$. A \emph{morphism of semilattices} is a function $f$ satisfying $f(\top)=\top$ and $f(s \wedge t) = f(s) \wedge f(t)$.
\end{definition}

We regard a semilattice as a category by letting elements be objects and having a unique morphism $s \to t$ when $s \leq t$, that is, when $s \wedge t = s$.
We will disregard size issues altogether; either by restricting to small categories throughout the article, or by allowing semilattices (and monoids) that are large -- the only place it seems to matter is Lemma~\ref{lem:samediff} below.
Recall that a \emph{groupoid} is a category whose every morphism is invertible.

\begin{definition}\label{def:orderedgroupoid}
  An \emph{ordered groupoid} is a groupoid internal to the category of partially ordered sets and monotone functions, together with a choice of \emph{restriction} $(f|A) \colon A \to B$ for each $f \colon A' \to B$ and $A \leq A'$ satisfying $(f|A)\leq f$. Explicitly, the sets $G_0$ and $G_1$ of objects and arrows are partially ordered, and the functions
  \[\begin{tikzpicture}[xscale=2.5,font=\small]
    \node (0) at (0,0){$G_0$};
    \node (1) at (1,0){$G_1$};
    \node (2) at (2,0){$G_2$};
    \draw[->] (0) to node[above=-.5mm]{id} (1);
    \draw[->] (1) to[out=-135,in=-45] node[below=-1mm]{cod} (0);
    \draw[->] (1) to[out=135,in=45] node[above]{dom} (0);
    \draw[->] (2) to node[above]{comp} (1);
    \draw[->] (1) to[out=60,in=0] +(0,.7) node[above]{inv} to[out=180,in=120] (1);
  \end{tikzpicture}\]
  are all monotone, where $G_2=\{(g,f) \in G_1^2 \mid \mathrm{dom}(g)=\mathrm{cod}(f) \}$ is ordered by $(g,f) \leq (g',f')$ when $g \leq g'$ and $f \leq f'$.
  An \emph{inductive groupoid} is an ordered groupoid whose partially ordered set of objects forms a semilattice.

  A morphism of ordered groupoids is a functor $F$ that is monotone in morphisms, that is, $F(f)\leq F(g)$ when $f \leq g$.
  Inductive groupoids and their morphisms form a category $\cat{IndGpd}$.
\end{definition}

\begin{theorem}\label{thm:esn}
  There is an equivalence $\cat{InvMon} \simeq \cat{IndGpd}$.
\end{theorem}
\begin{proof}[Proof sketch]
  See~\cite[Section~4.2]{lawson:inversesemigroups} or~\cite{dewolfpronk:esntheorem} for details.
  An inverse monoid $M$ turns into an inductive groupoid as follows.
  Objects are idempotents $ss^\dag=s \in M$.
  Every element of $M$ is a morphism $x \colon x^\dag x \to xx^\dag$. 
  The identity on $s$ is $s$ itself, and composition is given by multiplication in $M$.
  Inverses are given by $x^{-1}=x^\dag$.
  The order $x \leq y$ holds when $x=yx^\dag x$.
  The restriction of $x \colon x^\dag x \to xx^\dag$ to $s^\dag s = s \leq x^\dag x$ is $xs$.
\end{proof}

Observe from the proof of the previous theorem that commutative inverse monoids correspond to inductive groupoids where every morphism is an endomorphism. Moreover, the endohomsets are abelian groups. Hence commutative inverse monoids correspond to a semilattice of abelian groups.

\begin{definition}\label{def:slat}
  A \emph{semilattice over} a subcategory $\cat{V}$ of $\cat{Cat}$ is a functor $F \colon \cat{S}\op \to \cat{V}$ where $\cat{S}$ is a semilattice and all categories $F(s)$ have the same objects.
  A \emph{morphism of semilattices $F \to F'$ over $\cat{V}$} is a morphism of semilattices $\varphi \colon \cat{S} \to \cat{S'}$ together with a natural transformation $\theta \colon F \Rightarrow F' \circ \varphi$.
  Write $\cat{SLat}[\cat{V}]$ for the category of semilattices over $\cat{V}$ and their morphisms.
\end{definition}

The ordinary category of semilattices can be recovered by choosing $\cat{V}$ to be the category containing as its single object the terminal category $\cat{1}$. 
In the commutative case, the ESN theorem simplifies, as worked out by Jarek~\cite{jarek:semigroups}. The following formulation chooses $\cat{V}=\cat{Ab}$, regarding an abelian group as a one-object category.

\begin{theorem}\label{thm:jarek}
  If $M$ is a commutative inverse monoid, then
  \[
    \cat{S} = \{ s \in M \mid ss^\dag=s \}, \qquad s \wedge t = st, \qquad \top=1,
  \] 
  is a semilattice, and for each $s \in \cat{S}$, 
  \[
    F(s) = \{ x \in M \mid xx^\dag=s \}
  \]
  is an abelian group with multiplication inherited from $M$ and unit $s$,
  giving a semilattice of abelian groups $F \colon \cat{S} \to \cat{Ab}$ by $F(s \leq t)(x) \to sx$.

  If $F \colon \cat{S} \to \cat{Ab}$ is a semilattice of abelian groups, then 
  $M = \coprod_{s \in \cat{S}} F(s)$ is a commutative inverse monoid under
  \begin{align*}
    xy & = F(s \wedge t \leq s)(x) \cdot F(s \wedge t \leq t)(y) && \text{ if }x \in F(s),\ y \in F(t), \\
    x^\dag & = x^{-1} \in F(s) && \text{ if }x \in F(s), \\
    1 & = 1 \in F(\top).
  \end{align*}
  This gives an equivalence $\cat{cInvMon} \simeq \cat{SLat}[\cat{Ab}]$.  
\end{theorem}
\begin{proof}
  First, let $M$ be an inverse monoid.
  To see that $\cat{S}$ is a semilattice, it suffices to show that it is a commutative idempotent monoid. 
  Commutativity is inherited from $M$, and idempotence follows from the fact that $M$ is an inverse monoid: $(xx^\dag)^2=xx^\dag xx^\dag=xx^\dag$.
  Next we verify that each $F(s)$ is an abelian group. 
  It is closed under multiplication: if $x,y \in F(s)$, then $(xy)(xy)^\dag = xx^\dag y^\dag y = ss^\dag=s$ so also $xy \in F(s)$.
  It has $s$ as a unit: if $x \in F(s)$, then $sx=xx^\dag x=x$.
  The inverse of $x \in F(s)$ is given by $x^\dag$, because $xx^\dag=s$ by definition.
  Furthermore, the diagram $F$ is functorial: clearly $F(s \leq t) \circ F(r \leq s) (x)=Rx=F(r \leq t)(x)$, and $F(s \leq s)(x)=sx=xx^\dag x=x$. It is also well-defined: if $s \leq t$ and $x \in F(t)$, then $sx(sx)^\dag=sxx^\dag s^\dag = sts^\dag = ss^\dag=s$ so $sx \in F(t)$. 

  Now let $F \in \cat{SLat}[\cat{Ab}]$. Then $1 \in F(\top)$ acts as a unit in $M$: if $x \in F(s)$ then $x 1 = F(s \leq s)(x) \cdot F(s \leq \top)(1) = x \cdot 1 = x \in F(s)$. The multiplication is clearly associative and commutative, so $M$ is an abelian monoid. It is an inverse monoid because $xx^\dag x = xx^{-1}x = x$ is computed within $F(s)$.

  Next we move to morphisms. Given a morphism $f \colon M \to M'$ of commutative inverse monoids, define a morphism $F \to F'$ of their associated semilattices of abelian groups as follows: $\varphi \colon \cat{S} \to \cat{S'}$ is just $\varphi(s)=f(s)$, and $\theta_s \colon F(s) \to F'(f(s))$ is just $\theta_s(x)=f(x)$. This is clearly functorial $\cat{cInvMon} \to \cat{SLat}[\cat{Ab}]$.

  Conversely, given a morphism $(\varphi,\theta) \colon F \to F'$ of semilattices of abelian groups, define a homomorphism $M \to M'$ of their associated commutative inverse monoids by $F(s) \ni x \mapsto \theta_s(x) \in F(\varphi(s))$. This is clearly functorial $\cat{SLat}[\cat{Ab}] \to \cat{cInvMon}$.

  Finally, turning a commutative inverse monoid $M$ into a semilattice of abelian groups and that in turn into a commutative inverse monoid ends up with the exact same monoid $M$.
  A semilattice of abelian groups $F \colon \cat{S} \to \cat{Ab}$ gets mapped to the inverse monoid $\coprod_s F(s)$, which in turn gets mapped to the following semilattice of abelian groups $G \colon \cat{T} \to \cat{Ab}$. 
  The semilattice $\cat{T}$ is given by $\{t \in F(s) \mid s \in \cat{S}, t=tt^\dag\}=\{t \in F(s) \mid s \in \cat{S}, t=tt^{-1}=1\}=\{1 \in F(s) \mid s \in \cat{S}\}$; clearly $s \mapsto 1 \in F(s)$ is an isomorphism $\varphi \colon \cat{S} \to \cat{T}$.
  The abelian group $G(\varphi(s))$ is given by $\{ x \mid xx^\dag=s \} = \{x \mid 1=xx^{-1}=s\} = \{x \in F(s)\}$; clearly $x \mapsto x$ is a natural isomorphism $\theta_s \colon F(s) \to G(\varphi(s))$. Thus $G \simeq F$, and the two functors implement an equivalence.
\end{proof}

\section{Inverse categories}\label{sec:inversecategories}

This section extends the previous one to a typed setting. 
A \emph{dagger category} is a category with a contravariant involution $\dag$ that acts as the identity on objects. A \emph{dagger functor} is a functor between dagger categories satisfying $F(f^\dag)=F(f)^\dag$.
An \emph{inverse category} is a dagger category where $f = f f^\dag f$ and $ff^\dag gg^\dag = gg^\dag ff^\dag$ for any pair of morphisms $f$ and $g$ with the same domain~\cite{cockettlack:restriction}. Equivalently, it is a category where every morphism $f\colon A \to B$ allows a unique morphism $f^\dag \colon B \to A$ satisfying $f=ff^\dag f$ and $f^\dag=f^\dag f f^\dag$; thus every functor between inverse categories is in fact a dagger functor.
Inverse categories and (dagger) functors form a category $\cat{InvCat}$, and groupoids and functors form a full subcategory $\cat{Gpd}$.
The ESN theorem extends to inverse categories, as worked out by DeWolf and Pronk~\cite{dewolfpronk:esntheorem}.

\begin{definition}
  A \emph{locally complete inductive groupoid} is an ordered groupoid with a partition of the semilattice $G_0$ of objects into semilattices $\{M_i\}$ such that two objects are comparable if and only if they are in the same semilattice $M_i$.
  Locally complete inductive groupoids form a subcategory $\cat{lcIndGpd}$ of $\cat{IndGpd}$ of those functors that preserve greatest lower bounds of objects.
\end{definition}

\begin{theorem}\label{thm:dewolfpronk}
  There is an equivalence $\cat{InvCat}\simeq\cat{lcIndGpd}$.
\end{theorem}
\begin{proof}[Proof sketch]
  See~\cite{dewolfpronk:esntheorem} for details.
  An inverse category $\cat{C}$ turns into a locally complete inductive groupoid as follows.
  Objects are idempotents $ff^\dag$ for some endomorphism $f \colon A \to A$ in $\cat{C}$.
  These partition into the semilattices of idempotents on a fixed object $A$.
  Every morphism $f \colon A \to B$ of $\cat{C}$ becomes a morphism $f^\dag f \to ff^\dag$.
  The identity on $ff^\dag$ is $ff^\dag$ itself, and composition is inherited from $\cat{C}$.
  Inverses are given by $f^{-1}=f^\dag$.
  The order $f \leq g$ holds when $f=gf^\dag f$; clearly two identity morphisms are comparable exactly when they endomorphisms on the same object.
  The restriction of $f \colon f^\dag f \to ff^\dag$ to $s^\dag s=s \leq f^\dag f$ is $fs$.
\end{proof}

\begin{lemma}\label{lem:semilatofgroupoids}
  If $F \colon \cat{S}\op \to \cat{Gpd}$ is a semilattice of groupoids, there is a well-defined inverse category $\cat{C}$ with the same objects as $F(\top)$ and morphisms
  \[
    \cat{C}(A,B)=\coprod_{s \in \cat{S}} F(s)\big( A, B \big)\text.   
  \]
  If $(\varphi,\theta)$ is a morphism $F \to F'$ of semilattices of groupoids, then there is a dagger functor $\cat{C} \to \cat{C'}$ between their associated categories, given by $A \mapsto \theta_\top(A)$ on objects and $F(s) \ni f \mapsto \theta_s(f) \in F'(\varphi(s))$ on morphisms.
  This gives a functor $\cat{SLat}[\cat{Gpd}] \to \cat{InvCat}$.
\end{lemma}
\begin{proof}
  The composition of $f \in F(s)(A,B)$ and $g \in F(t)(A,B)$ is given by $F(s \wedge t \leq t)(g) \circ F(s \wedge t \leq s)(f) \in F(s \wedge t)(A,C)$; this is clearly associative.
  The identity on $A$ is given by $\id[A] \in F(\top)(A,A)$: if $f \in F(s)(A,B)$, then $f \circ \id[A] = F(s \wedge \top \leq \top)(\id[A]) \circ F(s \wedge \top \leq s)(f) = \id \circ F(s \leq s)(f)=f$.
  The dagger of $f \in F(s)(A,B)$ is given by $f^{-1} \in F(s)(B,A)$; this clearly is an inverse category.
\end{proof}

Combining Theorem~\ref{thm:dewolfpronk} and Lemma~\ref{lem:semilatofgroupoids}, we see that a semilattice of groupoids $F \colon \cat{S}\op \to \cat{Gpd}$ gives rise to a locally complete inductive groupoid $\cat{G}$ where:
\begin{itemize}
  \item objects are $\coprod_{A \in F(\top)} \coprod_{s \in \cat{S}} \{ f^\dag f \mid f \in F(s)(A,A) \}$;
  \item there is an arrow $(f^\dag f)_{A,s} \to (ff^\dag)_{B,s}$ for each $f \in F(s)(A,B)$;
  \item the composition of $f \in F(s)(A,B)$ and $g \in F(t)(B,C)$ is computed as $F(s \wedge t \leq t)(g) \circ F(s \wedge t \leq s)(f)$.
\end{itemize}
Not every locally complete inductive groupoid comes from a semilattice of groupoids in this way. 
Instead, locally complete inductive groupoids correspond to certain functors $\cat{S}\op \to \cat{Gpd}$ where $\cat{S}$ may be a disjoint union of several semilattices; a `multi-semilattice' of groupoids.

Notice that the objects of $\cat{G}$ are doubly-indexed: once by an object of the category $F(\top)$, and once by an element of the semilattice $\cat{S}$. 
Locally complete inductive groupoids and semilattices of groupoids have different ways of bookkeeping the same data, each emphasising one of these two indices.
In the remainder of the paper, we will prefer to work with semilattices of groupoids rather than the more general locally complete inductive groupoids for two reasons. 
First, the extra structure we will consider does not require `multi-semilattices', but instead is uniform enough so semilattices suffice.
Second, semilattices of groupoids form a purely categorical concept, whereas ordered groupoids require extra conditions on groupoids internal to the category of partially ordered sets that are somewhat ad hoc.
For example, this perspective will later enable us to remove the restriction that all groupoids in a semilattice of groupoids must have the same objects; see Lemma~\ref{lem:samediff} below.

\section{Compact inverse categories}\label{sec:compactinversecategories}

There is another way to categorify inverse monoids, that takes advantage of a degree of commutativity.
Instead of moving from inverse monoids to inverse categories, in this section we move to \emph{compact inverse categories}. The presence of the tensor product means that the latter specialise to commutative inverse monoids in the one-object case.
By a \emph{compact inverse category} we mean an inverse category that is also a compact dagger category under the same dagger~\cite{heunenvicary:cqm}. Here, a dagger category is compact when it is symmetric monoidal, $(f \otimes g)^\dag = f^\dag \otimes g^\dag$ for all morphisms $f$ and $g$, all coherence isomorphisms are inverted by their own daggers, and every object $A$ allows an object $A^*$ and a morphism $\eta_A \colon I \to A^* \otimes A$ satisfying
\begin{equation}\label{eq:snake}
  \id[A] = 
  \lambda_A \circ
  (\varepsilon \otimes \id[A]) \circ
  \alpha \circ
  (\id[A] \otimes \eta) \circ
  \rho_A^{-1}
\end{equation}
for $\varepsilon=\sigma \circ \eta^\dag$ where $\sigma$ is the swap map.
Let us first show that compact inverse categories indeed generalise commutative inverse monoids, because the property of compactness is hidden in the one-object case.

\begin{proposition}\label{prop:oneobject}
  One-object compact (dagger/inverse) categories are exactly commutative (involutive/inverse) monoids.  
\end{proposition}
\begin{proof}
  Let $M$ be a commutative  monoid. 
  Regard it as a one-object monoidal category. 
  The one object is the tensor unit, and in any monoidal category, the tensor unit $I$ is its own dual $I^*=I$, since $\eta=\lambda_I^{-1}$ and $\varepsilon=\rho_I$ satisfy~\eqref{eq:snake} by coherence~\cite[Lemma~3.6]{heunenvicary:cqm}.
  If the monoid is involutive/inverse, then the category is clearly dagger/inverse.

  Conversely, a one-object (dagger) category is clearly an (involutive) monoid.
  If the category is monoidal, then the monoid is necessarily that of scalars $I \to I$, where tensor and composition coincide and are commutative~\cite{abramsky:scalars}.
\end{proof}

We now set out to generalise Theorem~\ref{thm:jarek} to compact inverse categories $\cat{C}$. They have the right modicum of commutativity to take advantage of Lemma~\ref{lem:semilatofgroupoids}:
the monoid $\cat{C}(I,I)$ of \emph{scalars} is always commutative, 
any morphism $f \colon A \to B$ can be multiplied with a scalar $s \colon I \to I$ to give $s \bullet f = \lambda \circ (s \otimes f) \circ \lambda^{-1}$, 
and any endomorphism $f \colon A \to A$ has a \emph{trace} $\Tr(f) = \varepsilon \circ (f \otimes \id[A^*]) \circ \sigma \circ \eta \colon I \to I$.
Furthermore, any morphism $f \colon A \to B$ has a \emph{dual} $f^* = (\id[A^*] \otimes \varepsilon_B) \circ (\id[A^*] \otimes f \otimes \id[B^*]) \circ (\eta_A \otimes \id[B^*]) \colon B^* \to A^*$, satisfying $\Tr(f^*)=\Tr(f)^*$ when $A=B$.
We will write $\tr(f)$ instead of $\Tr(f)^*$.
The form of the following lemma resembles the categorical no-cloning theorem~\cite{abramsky:nocloning}, and is the heart of the matter.

\begin{lemma}\label{lem:endocollapse}
  In a compact inverse category, any endomorphism $f$ equals $\tr(f) \bullet \id$.
\end{lemma}
\begin{proof}
  Let $f \colon A \to A$ be an endomorphism.
  Compactness provides $\eta \colon I \to A^* \otimes A$ and $\varepsilon \colon A \otimes A^* \to I$ satisfying the snake equations. 
  In terms of $g=\varepsilon \otimes \id[A]$ and $h=\id[A] \otimes \eta^\dag = \id[A] \otimes (\varepsilon \circ \sigma)$, and suppressing coherence isomorphisms, these equations read $gh^\dag = \id[A] = hg^\dag$.
  It follows that
  \begin{align*}
    hh^\dag = gh^\dag hh^\dag = gh^\dag = \id[A]\text,\\
    g^\dag h = g^\dag g h^\dag h = h^\dag h g^\dag g = h^\dag g\text.
  \end{align*}
  Therefore $g = hh^\dag g = hg^\dag h = h$, and so
  \[
    f 
    = g \circ (\id[A] \otimes f^* \otimes \id[A]) \circ h^\dag
    = h \circ (\id[A] \otimes f^* \otimes \id[A]) \circ h^\dag
    = \Tr(f^*) \bullet \id[A]\text.\qedhere
  \]
\end{proof}

\begin{proposition}\label{prop:nonendo}
  A compact dagger category is a compact inverse category if and only if every morphism $f$ satisfies $f=\tr(f  f^\dag) \bullet f$.
\end{proposition}
\begin{proof}
  Suppose we're given a compact inverse category. By Lemma~\ref{lem:endocollapse}, the endomorphism $f  f^\dag$ equals $\tr(f f^\dag f f^\dag) \bullet \id = \tr(f f^\dag) \bullet \id$. Hence $f=\tr(f f^\dag) \bullet f$.

  Conversely, suppose given a compact dagger category in which every morphism satisfies $f=\tr(f f^\dag) \bullet f$. We will prove that this is a \emph{restriction category} with $\bar{f} = \tr(f f^\dag) \bullet \id$, by verifying the four axioms~\cite{cockettlack:restriction}.

  First, $f \bar{f} = \tr(f f^\dag) \bullet f = f$.
  Second, $\bar{f} \bar{g} = \tr(f f^\dag) \bullet \tr(g g^\dag) \bullet \id = \bar{g} \bar{f}$ if $\dom(f)=\dom(g)$.
  Third, 
  \begin{align}
	& \tr(ff^\dag)^\dag \circ \tr(ff^\dag) \notag\\
	& = (\varepsilon \otimes \varepsilon) \circ (\sigma \otimes \id) \circ (ff^\dag \otimes \id \otimes ff^\dag \otimes \id) \circ (\id \otimes \sigma)\circ (\eta \otimes \eta) \notag\\
	& = \varepsilon \circ (ff^\dag ff^\dag \otimes \id) \circ \sigma \circ \eta \notag\\
	& = \varepsilon \circ (ff^\dag \otimes \id) \circ \sigma \circ \eta \notag\\	
	& = \tr(ff^\dag) \tag{$*$}\label{eq:snaketrace}
  \end{align}
  by Lemma~\ref{lem:endocollapse}.
  Therefore, for $\dom(f)=\dom(g)$:
  \begin{align*}
	\overline{g  \bar{f}} 
	& = \overline{\tr(f f^\dag) \bullet g} \\
	& = \tr \big[ \tr(f f^\dag)^\dag \bullet \tr(f f^\dag) \bullet g g^\dag \big] \bullet \id \\
	& = \tr(f f^\dag)^\dag \bullet \tr(f f^\dag) \bullet \tr(g g^\dag) \bullet \id \\
	& = \tr(f f^\dag) \bullet \tr(g g^\dag) \bullet \id \\
	& = \bar{g} \bar{f}.
  \end{align*}
  Fourth,
  $\bar{g}  f
  = \tr(g g^\dag) \bullet f 
  = \tr(g g^\dag) \bullet \tr(f f^\dag) \bullet f$,
  and $f \overline{g f}
  = \tr(g f f^\dag g^\dag) \bullet f$. The two are equal by a similar computation as~\eqref{eq:snaketrace}.

  Finally, taking $g=f^\dag$ shows that $\bar{f}=\tr(f f^\dag) \bullet \id = g f$ by Lemma~\ref{lem:endocollapse} and similarly $\bar{g}=f g$. Therefore the category is compact inverse~\cite[Theorem~2.20]{cockettlack:restriction}.
\end{proof}

Next we build up to generalise Theorem~\ref{thm:jarek}, starting with the replacement for abelian groups. A \emph{compact groupoid} is a compact dagger category where any morphism $f$ is inverted by $f^\dag$.

\begin{lemma}\label{lem:cptgpd}
  Compact groupoids are precisely compact inverse categories with invertible scalars.
\end{lemma}
\begin{proof}
  Let $\cat{C}$ be a compact inverse category with invertible scalars.
  By Lemma~\ref{lem:endocollapse}, all endomorphisms are invertible.
  Let $f \colon A \to B$ be any morphism.
  Then $f f^\dag$ is an isomorphism, and so $f$ is (split) monic. Because $f=ff^\dag f$, it follows that $ff^\dag = \id[B]$.
  Similarly $f^\dag f$ is an isomorphism, so $f$ is (split) epic, whence $f^\dag f=\id[A]$.
  Thus $f$ is invertible.	
\end{proof}

We can now show that any compact inverse category is a semilattice of compact groupoids. Write $\cat{CptInvCat}$ for the category of compact inverse categories and (strong) monoidal dagger functors, and $\cat{CptGpd}$ for the full subcategory of compact groupoids and (strong) monoidal functors.

\begin{proposition}\label{prop:semilatofcptgpds}
  If $\cat{C}$ is a compact inverse category, then
  \[
    \cat{S} = \{ s \in \cat{C}(I,I) \mid ss^\dag=s \}, \qquad s \wedge t = st, \qquad \top = \id[I],
  \]
  is a semilattice, and for each $s \in \cat{S}$,
  there is a compact groupoid $F(s)$ with the same objects as $\cat{C}$ and morphisms
  \[
    F(s)(A,B) = \{ f \in \cat{C}(A,B) \mid \tr(ff^\dag)=s \}\text,
  \]
  giving a semilattice $F \colon \cat{S}\op \to \cat{CptGpd}$ of compact groupoids $F(s\leq t)(f) \mapsto s \bullet f$.

  The assignment $\cat{C} \mapsto F$ extends to a functor $\cat{CptInvCat} \to \cat{SLat}[\cat{CptGpd}]$ by sending a morphism $G \colon \cat{C} \to \cat{C'}$ to
  \[
    \varphi(s)=\psi_0^{-1} \circ G(s)\circ \psi_0, \qquad \theta_s(A)=G(A), \qquad \theta_s(f)=G(f)\text,
  \]
  where $\psi_0 \colon I' \to G(I)$ is the structure isomorphism.
\end{proposition}
\begin{proof}
  First, $\cat{S}$ is a commutative idempotent monoid by definition.

  Next, we verify that $F(s)$ is a compact groupoid.
  Composition is well-defined: if $f \colon A \to B$ and $g \colon B \to C$ satisfy $\tr(ff^\dag)=s=\tr(gg^\dag)$, then by Lemma~\ref{lem:endocollapse} and linearity and cyclicity of trace:
  \begin{align*}
    \tr\big((gf)(gf)^\dag\big)
    & = \tr(g^\dag g ff^\dag) \\
    & = \tr\big[ (\tr(g^\dag g) \bullet \id[B]) \circ (\tr(ff^\dag) \bullet \id[B]) \big] \\
    & = \tr(g^\dag g) \bullet \tr(ff^\dag) \bullet \tr(\id[B]) \\
    & = \tr(ff^\dag) \bullet \tr(\id[B]) \\
    & = \tr\big[ \id[B] \circ (\tr(ff^\dag) \bullet \id[B])] \\
    & = \tr(\id[B] \circ ff^\dag) \\
    & = \tr(ff^\dag) \\
    & = s\text.
  \end{align*}
  It is clear that $s \bullet \id[A]$ play the role of identities in $F(s)$.
  The category $F(s)$ is monoidal, because if $\tr(ff^\dag) = s = \tr(gg^\dag)$, then $\tr((f \otimes g) (f \otimes g)^\dag) = \tr(ff^\dag \otimes gg^\dag) = \tr(ff^\dag) \tr(gg^\dag) = s$. 
  It also inherits the dagger from $\cat{C}$: if $\tr(ff^\dag)=s$, then also $\tr(f^\dag f)=\tr(ff^\dag)=s$.
  Consequently, $F(s)$ inherits the property of being an inverse category from $\cat{C}$.
  Moreover, $F(s)$ is a compact dagger category: the units and counits are given by $s \bullet \eta_A$ and $s \bullet \varepsilon_A$.
  Finally, scalars $x \in F(s)(I,I)$ are those scalars $x \in \cat{C}(I,I)$ satisfying $x^\dag x=s$, and form an abelian group with inverse $x^\dag$ and unit $s$: for $xs=xx^\dag x=x$; if $x^\dag x=s=y^\dag y$ then $(xy)^\dag (xy)=x^\dag x y^\dag y = s^\dag s = s$; and $xx^\dag=s$.
  Lemma~\ref{lem:cptgpd} therefore makes $F(s)$ a compact groupoid.
  Notice that $F$ is a well-defined functor: if $s \leq t$ and $\tr(ff^\dag)=t$, then $st=t$, so $\tr((sf)(sf)^\dag) = s s^\dag \tr(ff^\dag) = st=s$.

  Now consider morphisms. If $\cat{G} \colon \cat{C} \to \cat{C'}$ is a monoidal dagger functor, say with structure isomorphisms $\psi_0 \colon I' \to G(I)$ and $\psi_{A,B} \colon G(A) \otimes' G(B) \to G(A \otimes B)$, then it is easy to see that $\varphi$ is a semilattice homomorphism, and that $\theta_s$ is a well-defined monoidal dagger functor that is moreover natural in $s$, because monoidal functors preserve dual objects and hence traces.
  Finally, it is clear that the assignment $G \mapsto (\varphi,f)$ is functorial.
\end{proof}



Notice that $\cat{S}$ contains all \emph{dimension} scalars $\dim(A)=\tr(\id[A])$.

\begin{lemma}\label{lem:semilatofcptgpds}
  If $F \colon \cat{S} \to \cat{CptGpd}$ is a semilattice of compact groupoids, then the category $\cat{C}$ of Lemma~\ref{lem:semilatofgroupoids} is a compact inverse category, and this gives a functor $\cat{SLat}[\cat{CptGpd}] \to \cat{CptInvCat}$.
\end{lemma}
\begin{proof}
  Define the tensor product on objects on $\cat{C}$ as in $F(\top)$, and set the tensor unit $I$ in $\cat{C}$ to be that of $F(\top)$. 
  The fact that $F(s \leq \top)$ are monoidal functors gives structure isomorphisms $\psi_s \colon A \otimes_s B \to A \otimes B$, where we write $\otimes_s$ for the tensor product in $F(s)$, and $\psi \colon I_s \to I$, where we write $I_s$ for the tensor unit in $F(s)$.
  Define the tensor product of $f \in F(s)(A,B)$ and $g \in F(t)(C,D)$ to be
  \[
    \psi_{s\wedge t} \circ \big( F(s \wedge t \leq s)(f) \otimes_{s \wedge t} F(s \wedge t \leq t)(g) \big) \circ \psi_{s\wedge t}^{-1}
  \]
  in $F(s \wedge t)\big( A \otimes C, B \otimes D \big)$.
  Taking coherence isomorphisms and dual objects as in $F(\top)$, a tedious but straightforward calculation proves that the triangle and pentagon axioms are satisfied, that the snake equations are satisfied, and that $\cat{C}$ is a compact inverse category. 

  An even more tedious but still straightforward calculation shows that the functor induced by a morphism of semilattices of compact groupoids is monoidal.
\end{proof}

\begin{theorem}\label{thm:semilatofcptgpds}
  The functors of Proposition~\ref{prop:semilatofcptgpds} and Lemma~\ref{lem:semilatofcptgpds} implement an equivalence $\cat{CptInvCat} \simeq \cat{SLat}[\cat{CptGpd}]$.
\end{theorem}
\begin{proof}
  Starting with a compact inverse category $\cat{C}$, turning it into a semilattice of compact groupoids $F$, and turning that into compact inverse category again, results in the exact same compact inverse category $\cat{C}$.
  For example, the old homset $\cat{C}(A,B)$ equals the new homset $\coprod_{s \in \cat{C}(I,I) \mid ss^\dag=s} \{ f \in \cat{C}(A,B) \mid \tr(ff^\dag)=s \}$ because any morphism $f$ in $\cat{C}$ is of the form $s \bullet f$ for some scalar $ss^\dag=s=\tr(ff^\dag)$ by Proposition~\ref{prop:nonendo}.
  Similarly, the new tensor product of $f \in F(s)(A,B)$ and $g \in F(t)(C,D)$ is 
  \begin{align*}
    & \psi_{s \wedge t} \circ \big( F(s \wedge t \leq s)(f) \otimes F(s \wedge t \leq s)(g) \big) \circ \psi_{s \wedge t}^{-1} \\
    & = \psi_{s \wedge t} \circ (stf \otimes stg) \circ \psi_{s \wedge t}^{-1} \\
    & = \psi_{s \wedge t} \circ (st \bullet (f \otimes g)) \circ \psi_{s \wedge t}^{-1} \\
    & = (st \bullet (f \otimes g)) \circ \psi_{s \wedge t} \circ \psi_{s \wedge t}^{-1} \\
    & = (s \bullet f)\otimes (t \bullet g) \\
    & = f \otimes g\text,
  \end{align*}
  again by Proposition~\ref{prop:nonendo}, and because the natural isomorphism $\psi$ cooperates with unitors and hence scalar multiplication, 
  and so equals the old tensor product.

  Now start with a semilattice of compact groupoids $F \colon \cat{S}\op \to \cat{CptGpd}$. Lemma~\ref{lem:semilatofcptgpds} turns it into a compact inverse category $\cat{C}$, which in turn becomes the following semilattice of compact groupoids $G \colon \cat{T}\op \to \cat{CptGpd}$. The semilattice $\cat{T}$ is
  \begin{align*}
    \coprod_{s \in \cat{S}} \{t \in F(s)(I,I) \mid tt^\dag=t\}
    = \coprod_{s \in \cat{S}} \{ \id[I] \in F(s)(I,I) \}
  \end{align*}
  because each $F(s)$ is a groupoid, so $s \mapsto \id[I] \in F(s)(I,I)$ is a semilattice isomorphism $\varphi \colon \cat{S} \to \cat{T}$.
  The construction of Proposition~\ref{prop:semilatofcptgpds} gives $G(\varphi(s))$ the same objects as $F(\top)$.
  Morphisms $A \to B$ in $G(\varphi(s))$ are $f\colon A \to B$ in $F(t)(A,B)$ for some $t \in \cat{S}$ satisfying $\varphi(s)=\tr(ff^\dag)$. Because $F(t)$ is a groupoid, $s$ must be $t$, so $G(\varphi(s))$ and $F(t)$ have the exact same homsets and identities, and we may take $\theta$ to be the identity functor. Going through the construction of $G$ shows that $\theta$ is in fact a monoidal dagger functor. 
\end{proof}

\section{examples}\label{sec:examples}

This section lists examples of compact inverse categories $\cat{C}$. For each example we will indicate how Proposition~\ref{prop:semilatofcptgpds} works by writing $\cat{C}_0$ for the semilattice $\cat{S}$ and $\cat{C}_s$ for the compact groupoid $F(s)$.

\begin{example}[The fundamental compact groupoid]
  Any topological space $X$ with a fixed chosen point $x \in X$	gives rise to a compact groupoid $\cat{C}$:
  \begin{itemize}
  	\item The objects of $\cat{C}$ are paths from $x_0$ to $x_0$, more precisely, continuous functions $f \colon [0,1] \to X$ with $f(0)=f(1)=x$.
  	\item The arrows $f \to g$ are homotopy classes of paths, more precisely, continuous functions $h \colon [0,1]^2 \to X$ such that $h(s,0)=f(s)$, $h(s,1)=g(s)$, and $h(0,t)=h(1,t)=x_0$, where $h$ and $h'$ are identified when there is a continuous function $H \colon [0,1]^3 \to X$ with $H(s,t,0)=h(s,t)$, $H(s,t,1)=h'(s,t)$, $H(s,0,u)=f(s)$, $H(s,1,u)=g(s)$, and $H(0,t,u)=H(1,t,u)=x_0$.
  	\item The tensor product of objects is composition of paths according to some fixed reparametrisation, the tensor unit is the constant path. Reparametrisation leads to associators and unitors. 
  	\item Dual objects are given by reversal of paths.
  	\item The dagger is given by reversal of homotopies.
  	\item The unit $\eta_f$ is the ``birth of a double loop'', a homotopy that ``grows'' from the constant path to the path $f^\dag \circ f$ by travelling progressively further along $f$ before travelling back along $f^\dag$.
  	\item The counit $\varepsilon_f$ is the ``contraction of a double loop'', a homotopy that ``shrinks'' from the path $f^\dag \circ f$ to the constant path.
  \end{itemize}
  In this case $\cat{C}_0$ is a one-element semilattice, and $\cat{C}_s=\cat{C}$ is already a groupoid.
\end{example}

\begin{example}
  Any abelian group $\cat{C}$, considered as a discrete monoidal category, is a compact groupoid.
  In this case $\cat{C}_0$ is a one-element semilattice, and $\cat{C}_s=\cat{C}$ is already a groupoid.
\end{example}

\begin{lemma}
  If $\cat{C}$ is a compact (dagger/inverse) category, and $S$ a family of (dagger) idempotents, then $\mathrm{Split}_S(\cat{C})$ is again (dagger/inverse) compact.
\end{lemma}
\noindent In terms of Theorem~\ref{thm:semilatofcptgpds}, $\mathrm{Split}_S(\cat{C})_0 \simeq \cat{C}_0$, and $\mathrm{Split}_S(\cat{C})_s = \mathrm{Split}_{S_s}(\cat{C}_s)$, where $S_s = \{ p \in S \mid \tr(p)=s \}$.
\begin{proof}
  Let $p \colon A \to A$ be in $S$.
  Define $\eta_p = (p^* \otimes p) \circ \eta_A \colon \id[I] \to p \otimes p^*$ and $\varepsilon_p = \varepsilon_A \circ (p \otimes p^*) \colon p^* \otimes p \to \id[I]$; these are well-defined morphisms in $\mathrm{Split}_S(\cat{C})$.
  Then indeed the snake equations hold: $p = (\varepsilon_A \otimes p) \circ (p \otimes p^* \otimes p) \circ (p \otimes \eta_A) = (\varepsilon_p \otimes p) \circ (p \otimes \eta_p)$.
  If $\cat{C}$ has a dagger, then so does $\mathrm{Split}_S(\cat{C})$, and $\eta_p = (\varepsilon_p \circ \sigma)^\dag$.
\end{proof}


\begin{example}
  If $\cat{C}$ and $\cat{D}$ are compact inverse categories, then so is $\cat{C} \times \cat{D}$.
  In this case $(\cat{C} \times \cat{D})_0 \simeq \cat{C}_0 \times \cat{D}_0$, and
  $(\cat{C} \times \cat{D})_{(s,t)} = \cat{C}_s \times \cat{D}_t$.
  If $\cat{C}$ and $\cat{D}$ are compact groupoids, then so is $\cat{C} \times \cat{D}$.
\end{example}

\begin{example}
  If $\cat{C}$ is a compact inverse category, and $\cat{G}$ is a groupoid, then $[\cat{G},\cat{C}]_\dag$, the category of functors $F \colon \cat{G} \to \cat{C}$ satisfying $F(f^{-1})=F(f)^\dag$ and natural transformations, is again a compact inverse category. 
\end{example}
\noindent
In this case $([\cat{G},\cat{C}]_\dag)_0 \simeq \cat{C}_0$, and
$([\cat{G},\cat{C}]_\dag)_s$ has as morphisms natural transformations whose every component is in $\cat{C}_s$.
\begin{proof}
  If $\alpha \colon F \Rightarrow G$ is a natural transformation, its dagger is given by $(\alpha^\dag)_A = (\alpha_A)^\dag \colon G(A) \to F(A)$; naturality of $\alpha^\dag$ follows from naturality of $\alpha$ together with the conditions $F(f)^\dag=F(f^{-1})$ and $G(f)^\dag=G(f^{-1})$. This makes $[\cat{G},\cat{C}]_\dag$ into a dagger category. 
  It inherits the property $\alpha=\alpha \alpha^\dag \alpha$ componentwise from $\cat{C}$, and is therefore an inverse category.

  The tensor product of objects is given by $(F \otimes G)(A) = F(A) \otimes G(A)$, and on morphisms by $(F \otimes G)(f) = F(f) \otimes G(f)$. The tensor unit is the functor that is constantly $I$. Because the coherence isomorphisms in $\cat{C}$ are unitary, this makes $[\cat{G},\cat{C}]_\dag$ into a well-defined dagger symmetric monoidal category.

  Finally, the dual object of $F \colon \cat{G} \to \cat{C}$ is given by $F^*(A)=F(A)^*$ and $F^*(f)=F(f)_*$. The unit $\eta_F \colon I \Rightarrow F^* \otimes F$ is given by $(\eta_F)_A = \eta_{F(A)}$, and the counit by $(\varepsilon_F)_A=\varepsilon_{F(A)}$. These are natural because any morphism $f \colon A \to B$ in $\cat{G}$ satisfies $ff^\dag = \id[A]$, whence $(F(f)_* \otimes F(f)) \circ \eta_{F(A)} = (\id[B^*] \otimes f) \circ (\id[B^*] \otimes f^\dag) \circ \eta_{F(B)} = \eta_{F(B)}$. This makes $[\cat{G},\cat{C}]_\dag$ a compact inverse category.
\end{proof}

\section{Compact groupoids}\label{sec:compactgroupoids}

This section moves to a 2-categorical perspective, to connect to a characterisation of compact groupoids.
A compact groupoid is the same thing as a \emph{coherent 2-group}~\cite{baezlauda:2groups}.
It is also known as a \emph{crossed module}.
Compact groupoids are classified by two abelian groups $G$ and $H$ and an element of the third cohomology group of $G$ with coefficients in $H$, as worked out by Baez and Lauda~\cite{baezlauda:2groups}. The following proposition makes this more precise. In the nonsymmetric case, $G$ need not be abelian, and there is an additional action of $G$ on $H$.

\begin{proposition}
  A compact groupoid $G$ is, up to equivalence, defined by the following data:
  \begin{itemize}
  \item the (abelian) group $G$ of isomorphism classes of objects of $\cat{C}$, under $\otimes$, with unit $I$, and inverse given by dual objects;
  \item the abelian group $H$ of scalars $\cat{C}(I,I)$ under composition with unit $\id[I]$ and inverse $\dag$;
  \item the conjugation action $G \times H \to H$ that takes $(A,s)$ to $\tr(A \otimes s)=s$;
  \item the 3-cocycle $G \times G \times G \to H$ that takes $(A,B,C)$ to $\Tr(\alpha_{A,B,C})$.
  \end{itemize}
\end{proposition}
The above data form the objects of a (weak) 2-category $\cat{Cocycle}$, with 1- and 2-cells as in~\cite[Theorem~43]{baezlauda:2groups}.
\begin{proof}[Proof sketch]
  See~\cite[Section~8]{baezlauda:2groups}.
  The trick is the following. First, we may assume that $\cat{C}$ is skeletal.
  Then, we may adjust the tensor product such that all unitors and units and counits (but not the associators!) are identities.
  The pentagon equation ensures that the trace of the associator is in fact a 3-cocycle.
\end{proof}

The proof of Theorem~\ref{thm:semilatofcptgpds} is the only place where we have used that in a semilattice $F$ of categories all $F(s)$ must have the same objects. It was needed because if the functor $\theta_s$ is to be an isomorphism, it must give a bijection between the objects of $F(s)$ and $F(\top)$. 
We now move to a (weak) 2-categorical perspective to remove this restriction.

\begin{definition}\label{def:slat2}
  Redefine the category $\cat{SLat}[\cat{V}]$ of Definition~\ref{def:slat} to become a (weak) 2-category as follows:
  \begin{itemize}
    \item 0-cells are functors $F \colon \cat{S}\op \to \cat{Cat}$ for some semilattice $\cat{S}$;
    \item 1-cells $F \to F'$ consist of a morphism $\varphi \colon \cat{S} \to \cat{S'}$ of semilattices and a natural transformation $\theta \colon F \To F' \circ \varphi$;
    \item 2-cells $(\varphi,\theta) \to (\varphi',\theta')$ exist when $\varphi \leq \varphi'$ and then are natural transformations $\gamma \colon \theta \TO \theta' \circ (\id * (\varphi \leq \varphi'))$.
  \end{itemize}
  Composition is by pasting.
  \[\begin{tikzpicture}
    \node (J) at (0,1) {$\cat{S}\op$};
    \node (J') at (0,-1) {$\cat{S'}\op$};
    \node (C) at (5,0) {$\cat{V}$};
    \draw[->] (J) to[out=0,in=90,looseness=.5] node[above]{$F$} (C);
    \draw[->] (J') to[out=0,in=-90,looseness=.5] node[below]{$F'$} (C);
    \draw[->] (J) to[out=-120,in=120] node[left]{$\varphi$} (J');
    \draw[->] (J) to[out=-60,in=60] node[right]{$\varphi'$} (J');
    \node at (0,0) {$\leq$};
    \draw[twocell] (3,.5) to[out=-60,in=60] node[right]{$\theta'$} (3,-.5);
    \draw[twocell] (2,.5) to[out=-120,in=120] node[left]{$\theta$} (2,-.5);
    \draw[threecell] (2.2,0) to node[above]{$\gamma$} (2.8,0);
  \end{tikzpicture}\]
  Write $\cat{SLat}_=[\cat{CptGpd}]$ for the full sub-2-category where all categories $F(s)$ have the same objects.
\end{definition}

To be precise, in $\cat{SLat}[\cat{CptGpd}]$, 2-cells $\gamma$ are modifications: for each $s \in \cat{S}$ and $A \in F'(\varphi(s))$, there is a morphism $\gamma_{s,A} \colon \theta_s(A) \to \theta'_s\big(F'\big(\varphi(s) \leq \varphi'(s)\big)(A)\big)$ that is natural in $s$ as well as $A$.

\begin{lemma}\label{lem:samediff}
  There is a (weak) 2-equivalence $\cat{SLat}[\cat{CptGpd}] \simeq \cat{SLat}_=[\cat{CptGpd}]$.
\end{lemma}
\begin{proof}
  First, observe that two 0-cells $F,G \colon \cat{S}\op \to \cat{CptGpd}$ are equivalent in $\cat{SLat}[\cat{CptGpd}]$ exactly when there is a natural monoidal equivalence $F(s) \simeq G(s)$. 
  Therefore, it suffices to construct, for each $F$, such a $G$ such that each $G(s)$ has the same objects.
  Let $\kappa_s$ be the cardinality of the objects of $F(s)$, and let $\kappa$ be the maximum of all $\kappa_s$. 
  Define $G(s)$ to be equal to $F(s)$, except that we add $\kappa$ isomorphic copies of the tensor unit $I$.
  There is an obvious monoidal structure on $G(s)$, and by construction there is a monoidal equivalence $F(s) \simeq G(s)$, so that $G(s)$ is automatically a compact groupoid.
  We may furthermore relabel the objects of $G(s)$ to be ordinal numbers, so that all $G(j)$ have the same objects.
\end{proof}

\begin{theorem}
  There is a (weak) 2-equivalence $\cat{CptInvCat} \simeq \cat{SLat}[\cat{Cocycle}]$, where $\cat{CptInvCat}$ has natural transformations as 2-cells.  
\end{theorem}
\begin{proof}
  The (weak) 2-equivalence $\cat{CptGpd}\simeq \cat{Cocycle}$ of~\cite[Theorem~43]{baezlauda:2groups} induces a (weak) 2-equivalence $\cat{SLat}[\cat{CptGpd}] \simeq \cat{SLat}[\cat{Cocycle}]$ by postcomposition.
  Combine this with the equivalence $\cat{SLat}_=[\cat{CptGpd}] \simeq \cat{SLat}[\cat{CptGpd}]$ of Lemma~\ref{lem:samediff} and the equivalence $\cat{CptInvCat} \simeq \cat{SLat}_=[\cat{CptGpd}]$ of Theorem~\ref{thm:semilatofcptgpds}; the latter still holds after the change of Definition~\ref{def:slat2}.
\end{proof}

\section{Concluding remarks}\label{sec:conclusion}

We conclude by discussing the many questions left open and raised in this paper.
First, one could investigate generalising the results in this paper from categories to semicategories.
Second, one could investigate generalising the results in this paper from compact categories to monoidal categories where every object has a dual.

\subsection{Traced inverse categories}

Inverse categories provide semantics for reversible programs, but higher-order aspects of reversible programming remain unclear. Compact categories are closed and hence provide semantics for higher-order programming. Theorem~\ref{thm:semilatofcptgpds} shows that compact inverse categories are, in a sense, degenerate. But one of the most interesting aspects of higher-order programming, tail recursion, doesn't need compact categories for semantics, and can already be modeled in traced monoidal categories. (But see also~\cite{kaarsgaardaxelsengluck:joininversecats}.) Now every traced monoidal category can be monoidally embedded in a compact category~\cite{joyalstreetverity:traced}. One can prove that there exists a left dagger biadjoint to the forgetful functor from dagger compact categories to dagger traced categories. There is also a left adjoint to the forgetful functor from compact inverse categories to compact dagger categories, but the latter is not faithful. Hence there is a left dagger biadjoint to the forgetful functor from compact inverse categories to traced inverse categories, but its unit does not embed any traced inverse category into a compact inverse category. Therefore Theorem~\ref{thm:semilatofcptgpds} does not show that all traced inverse categories degenerate. Indeed, the category $\cat{PInj}$ of sets and injections is the universal inverse category~\cite{kastl:inversecategories}, and is also traced~\cite{hines:thesis,haghverdiscott:goi}, but it fails Lemma~\ref{lem:endocollapse}, irrespective of which tensor product it carries, as the swap map on the two element set is not a scalar multiple of the identity. That leaves a valid question: what do traced inverse categories look like?

\subsection{Idempotents}

A \emph{subunit} in a monoidal category $\cat{C}$ is a subobject $r \colon R \rightarrowtail I$ for which $r \otimes \id[R]$ is invertible~\cite{enriquemolinerheunentull:tensortopology}; they form a semilattice $\mathrm{ISub}(\cat{C})$. The following lemma shows that in compact inverse categories, up to splitting idempotents, the semilattice $\cat{C}_0$ is precisely that of subunits. See also~\cite{schwabschwab:idempotents} for structure theorems of inverse categories in which all idempotents split.

\begin{lemma}
  Let $\cat{C}$ be a compact inverse category.
  \begin{enumerate}
    \item[(a)] A map $r \colon R \to I$ is a subunit if and only if $r^\dag r = \id$.
    \item[(b)] Any subunit $r$ induces an element $rr^\dag$ of $\cat{C}_0$.
    \item[(c)] If idempotents split in $\cat{C}$, any element $\cat{C}_0$ is $rr^\dag$ for a unique subunit $r$; this gives a isomorphism between the semilattices $\cat{C}_0$ and $\mathrm{ISub}(\cat{C})$.
  \end{enumerate}
\end{lemma}
\begin{proof}
  For (a), first notice that if $r \colon R \to I$ is monic, then because $r=rr^\dag r$ in fact $r$ is an isometry, that is, $r^\dag r = \id$. We will show that for isometries $r$, the condition that $r \otimes \id[R]$ is invertible holds automatically, with the inverse being $r^\dag \otimes \id[R]$. It suffices to show that $(r \otimes \id[R])(r^\dag \otimes \id[R])=\id[I \otimes R]$. But 
  \[
    \id[I \otimes R] = \id[I] \otimes (r^\dag r) = \id[I] (r^\dag (r r^\dag) r) = (r r^\dag) \otimes (r^\dag r) = (r r^\dag) \otimes \id[R]\text.
  \]
  Thus the subunits are precisely the (subobjects represented by) isometries.

  Part (b) is obvious: if $r$ is an isometry, then $s=rr^\dag \colon I \to I$ satisfies $s=ss^\dag$.

  Part (c) follows from~\cite[Lemma~2.25]{cockettlack:restriction}, as does the fact that the maps of (b) and (c) are each other's inverses. It is easy to see that both maps preserve the order structure using~\cite[Proposition~2.8]{enriquemolinerheunentull:tensortopology}.
\end{proof}

Now there are two ways to `localise' $\cat{C}$ to $r \in \mathrm{ISub}(\cat{C})$. 
The localisation $\cat{C}\big|_r$ according to~\cite{enriquemolinerheunentull:tensortopology} has objects $A$ such that $r \otimes \id[A]$ is invertible, and all morphisms between those objects.
The localisation $\cat{C}_{rr^\dag}$ above has all objects, but only those morphisms $f$ satisfying $\tr(ff^\dag)=rr^\dag$.
These two localisations are different. The former localises with respect to the tensor product, whereas the latter localises with respect to composition.

Generally, taking semilattices of categories is a completion procedure. Does it generalise to (weak) 2-categories? If so, the above may be the special cases of a single object and of unique 2-cells, and could form a higher-categorical analogue of the Eckmann-Hilton argument in the Baez-Dolan stabilisation hypothesis~\cite{baezdolan}. Is there a relationship with~\cite{hayashi:semifunctors}?

\subsection{Internal descriptions}

Groupoids are precisely special dagger Frobenius algebras in the category $\cat{Rel}$ of sets and relations~\cite{heunencontrerascattaneo:groupoids}. 
Compact groupoids are precisely special dagger Frobenius algebras in the category $\cat{Rel}(\cat{Gp})$ of relations over the regular category of groups, see~\cite{granheunentull:connectors}.
Can inverse categories similarly be described as certain monoids in a category of relations?

\subsection{Bratteli diagrams and C*-algebras}

Describing compact inverse categories through a diagram of groupoids resembles describing an AF C*-algebra as a diagram of finite-dimensional C*-algebras~\cite{aminielliotgolestani:bratteli}. It is very fruitful to work with this so-called Bratteli diagram directly rather than with the C*-algebra itself. 
More generally, inverse semigroups are a popular way to generate C*-algebras~\cite{duncanpaterson:cstaralgebras}, as it is easier to work with the inverse semigroup directly, and moreover this captures many important classes of C*-algebras~(see \textit{e.g.}~\cite{starling:booleaninversemonoids}): AF C*-algebras, graph C*-algebras, tiling C*-algebras, self-similar group C*-algebras, subshift C*-algebras, C*-algebras of ample {\'e}tale groupoids, and C*-algebras of Boolean dynamical systems. There is also a multiply-typed version building a C*-algebra from a so-called higher rank graph~\cite{kumjianpask:higherrankgraph}. Can one similarly generate a C*-algebra from a compact inverse category, and is there a relationship to these other constructions?
A first step might be to extend~\cite{linckelmann:transfer} to possibly infinite categories by adding a norm.

\bibliographystyle{plain}
\bibliography{cptinv}

\begin{thebibliography}{10}

\bibitem{abramsky:scalars}
S.~Abramsky.
\newblock Abstract scalars, loops, and free traced and strongly compact closed
  categories.
\newblock In {\em Conference on Algebra and Coalgebra}, volume 3629 of {\em
  Lecture Notes in Computer Science}, pages 1--31. Springer, 2005.

\bibitem{abramsky:nocloning}
S.~Abramsky.
\newblock No-cloning in categorical quantum mechanics.
\newblock In {\em Semantic techniques for quantum computation}, pages 1--28.
  Cambridge University Press, 2008.

\bibitem{aminielliotgolestani:bratteli}
M.~Amini, G.~A. Elliott, and N.~Golestani.
\newblock The category of {B}ratteli diagrams.
\newblock {\em Canadian Journal of Mathematics}, 67:990--1023, 2015.

\bibitem{baezdolan}
J.~C. Baez and J.~Dolan.
\newblock Higher-dimensional algebra and topological quantum field theory.
\newblock {\em Journal of Mathematical Physics}, 36:6073--6105, 1995.

\bibitem{baezlauda:2groups}
J.~C. Baez and A.~Lauda.
\newblock Higher-dimensional algebra {V}: 2-groups.
\newblock {\em Theory and Applications of Categories}, 12:423--491, 2004.

\bibitem{cockettlack:restriction}
J.~R.~B. Cockett and S.~Lack.
\newblock Restriction categories {I}: categories of partial maps.
\newblock {\em Theoretical Computer Science}, 270(1--2):223--259, 2002.

\bibitem{dewolfpronk:esntheorem}
D.~De{W}olf and D.~Pronk.
\newblock The {E}hresmann-{S}chein-{N}ambooripad theorem for inverse
  categories.
\newblock {\em Theory and Applications of Categories}, 33(27):813--831, 2018.

\bibitem{duncanpaterson:cstaralgebras}
J.~Duncan and A.~L.~T. Paterson.
\newblock {C*}-algebras of inverse semigroups.
\newblock {\em Proceedings of the Edinburgh Mathematical Society}, 28:41--58,
  1985.

\bibitem{ehresmann:gattungen}
C.~Ehresmann.
\newblock Gattungen von lokalen {S}trukturen.
\newblock {\em Jahresbericht der Deutschen Mathematiker-Vereinigung},
  60:49--77, 1958.

\bibitem{ehresmann:inductives}
C.~Ehresmann.
\newblock Cat{\'e}gories inductives et pseudogroupes.
\newblock {\em Annales de l'{I}nstitut {F}ourier}, 10:307--332, 1960.

\bibitem{enriquemolinerheunentull:tensortopology}
P.~{Enrique Moliner}, C.~Heunen, and S.~Tull.
\newblock Tensor topology.
\newblock {\em arXiv:1810.01383}, 2018.

\bibitem{giles:thesis}
B.~G. Giles.
\newblock {\em An investigation of some theoretical aspects of reversible
  computing}.
\newblock PhD thesis, University of Calgary, 2014.

\bibitem{granheunentull:connectors}
M.~Gran, C.~Heunen, and S.~Tull.
\newblock Monoidal characterisation of groupoids and connectors.
\newblock {\em Topology and Applications}, 2019.

\bibitem{haghverdiscott:goi}
E.~Haghverdi and P.~Scott.
\newblock A categorical model for the geometry of interaction.
\newblock {\em Theoretical Computer Science}, 350:252--274, 2006.

\bibitem{hayashi:semifunctors}
S.~Hayashi.
\newblock Adjunction of semifunctors: categorical structures in nonextensional
  lambda calculus.
\newblock {\em Theoretical Computer Science}, 41:95--104, 1985.

\bibitem{heunencontrerascattaneo:groupoids}
C.~Heunen, I.~Contreras, and A.~S. Cattaneo.
\newblock Relative {F}robenius algebras are groupoids.
\newblock {\em Journal of Pure and Applied Algebra}, 217:114--124, 2013.

\bibitem{heunenvicary:cqm}
C.~Heunen and J.~Vicary.
\newblock {\em Categories for Quantum Theory}.
\newblock Oxford University Press, 2019.

\bibitem{hines:thesis}
P.~Hines.
\newblock {\em The algebra of self-similarity and its applications}.
\newblock PhD thesis, University of Wales, 1997.

\bibitem{jarek:semigroups}
P.~Jarek.
\newblock Commutative regular semigroups.
\newblock {\em Colloquium Mathematicum}, XII(2):195--208, 1964.

\bibitem{joyalstreetverity:traced}
A.~Joyal, R.~Street, and D.~Verity.
\newblock Traced monoidal categories.
\newblock {\em Math. Proc. Camb. Phil. Soc.}, 119(3):447--468, 1996.

\bibitem{kaarsgaardaxelsengluck:joininversecats}
R.~Kaarsgaard, H.~B. Axelsen, and R.~Gl{\"u}ck.
\newblock Join inverse categories and reversible recursion.
\newblock {\em Journal of Logical and Algebraic Methods in Programming},
  87:33--50, 2017.

\bibitem{kastl:inversecategories}
J.~Kastl.
\newblock {\em Algebraische {M}odelle, {K}ategorien und {G}ruppoide}, chapter
  Inverse categories, pages 51--60.
\newblock Number~7 in Studien zur {A}lgebra und ihre {A}nwendungen.
  Akademie-{V}erlag, 1979.

\bibitem{kumjianpask:higherrankgraph}
A.~Kumjian and D.~Pask.
\newblock Higher rank graph {C*}-algebras.
\newblock {\em New York Journal of Mathematics}, 6:1--20, 2000.

\bibitem{lawson:inversesemigroups}
M.~V. Lawson.
\newblock {\em Inverse semigroups}.
\newblock World Scientific, 1998.

\bibitem{linckelmann:transfer}
M.~Linckelmann.
\newblock On inverse categories and transfer in cohomology.
\newblock {\em Proceedings of the Edinburgh Mathematical Society}, 56:187--210,
  2013.

\bibitem{nambooripad:regularsemigroups}
K.~S.~S. Nambooripad.
\newblock {\em Structure of regular semigroups}.
\newblock American Mathematical Society, 1979.

\bibitem{schein:inversesemigroups}
B.~M. Schein.
\newblock On the theory of inverse semigroups and generalised groups.
\newblock {\em American Mathematical Society Translations}, 113(2):89--122,
  1979.

\bibitem{schwabschwab:idempotents}
E.~Schwab and E.~D. Schwab.
\newblock On inverse categories with split idempotents.
\newblock {\em Archivum Mathematicum}, 51:13--25, 2015.

\bibitem{seely:linearlogic}
R.~A.~G. Seely.
\newblock Linear logic, *-autonomous categories and cofree coalgebras.
\newblock {\em Contemporary Mathematics}, 92:371--382, 1989.

\bibitem{starling:booleaninversemonoids}
C.~Starling.
\newblock {C*}-algebras of {B}oolean inverse monoids -- traces and invariant
  means.
\newblock {\em Documenta Mathematica}, 21:809--840, 2016.

\end{thebibliography}

\end{document}